\documentclass[10pt,a4paper,reqno]{amsart}
\usepackage[english]{babel}
\usepackage[T1]{fontenc}

\usepackage[dvipsnames]{xcolor}
\usepackage[normalem]{ulem}
\usepackage
[colorlinks=true,linkcolor=Maroon,citecolor=OliveGreen,backref]
{hyperref}
\usepackage[abbrev,shortalphabetic]{amsrefs}  
\usepackage{cleveref}

\usepackage[norefs, nocites]{refcheck} 

\makeatletter
\newcommand{\refcheckize}[1]{%
  \expandafter\let\csname @@\string#1\endcsname#1%
  \expandafter\DeclareRobustCommand\csname relax\string#1\endcsname[1]{%
    \csname @@\string#1\endcsname{##1}\@for\@temp:=##1\do{\wrtusdrf{\@temp}\wrtusdrf{{\@temp}}}}%
  \expandafter\let\expandafter#1\csname relax\string#1\endcsname
}
\makeatother
\refcheckize{\cref}
\refcheckize{\Cref}

\newtheorem{theorem}{Theorem}
\newtheorem{lemma}[theorem]{Lemma}
\newtheorem{prop}[theorem]{Proposition}
\newtheorem{cor}[theorem]{Corollary}

\theoremstyle{definition}

\theoremstyle{remark}
\newtheorem{remark}[theorem]{Remark}

\newcommand {\Aut}[1]{{\sf Aut}({#1})}

\newcommand{\cT}{{\mathcal T}}

\newcommand{\cC}{{\mathcal C}}

\newcommand{\cH}{{\mathcal H}}
\newcommand{\cJ}{{\mathcal J}}

\newcommand{\comment}[1]{}

\newcommand{\fS}{\mathfrak S}
\newcommand{\lt}{\lambda} 
\newcommand{\wt}{\operatorname{wt}} 
\newcommand{\supp}{\mathrm{supp}}

\def\X{\mathfrak{X}}
\def\Y{\mathfrak{Y}}
\def\T{\mathcal{T}}
\renewcommand{\subset}{\subseteq}
\renewcommand{\supset}{\supseteq}

\title{Fusions of tensor powers of Johnson schemes}

\author{Sean Eberhard}
\address{Sean Eberhard, Centre for Mathematical Sciences, Wilberforce Road, Cambridge CB3~0WB, UK}
\email{eberhard@maths.cam.ac.uk}

\author{Mikhail Muzychuk}
\address{Mikhail Muzychuk, Department of Mathematics, Ben Gurion University
P.O.B. 653, Beer Sheva 8410501, Israel}
\email{\textbf{muzychuk@bgu.ac.il}}

\thanks{SE has received funding from the European Research Council (ERC) under the European Union's Horizon 2020 research and innovation programme (grant agreement No. 803711).}

\renewcommand{\textcolor}[2]{#2}

\begin{document}

\begin{abstract}
This paper is a follow-up to \cite{E}, in which the first author
studied primitive association schemes lying between a tensor power $\cT_m^d$ of the trivial association scheme and the Hamming scheme $\cH(d,m)$.
A question which arose naturally in that study was whether all primitive fusions of $\cT_m^d$ lie between $\cT_{m^e}^{d/e}$ and $\cH(d/e, m^e)$ for some $e \mid d$.
This note answers this question positively provided that $m$ is large enough.
We similarly classify primitive fusions of the $d$th tensor power of a Johnson scheme on $\binom{m}{k}$ points provided $m$ is large enough in terms of $k$ and $d$.
\end{abstract}

\maketitle

\section{Introduction}

Association schemes are objects of central importance in algebraic combinatorics.
The reader needing an introduction to association schemes could refer to any of \cites{CP,BI,Z,B} or the introduction to \cite{E}.

All our association schemes are symmetric.
If $\X, \Y$ are association schemes on a common vertex set then we write $\X \le \Y$ if $\X$ refines $\Y$ as a partition.
A fusion of an association scheme is a coarsening which is again an association scheme.
\textcolor{blue}{Notice that some authors, for example \cite{KK}, use an opposite order on the set of association schemes. Our choice was motivated by keeping notation consistent with \cite{B,E}. Notice that with this choice of ordering the inclusion $\X \le \Y$ implies a similar inclusion between the automorphism groups $\Aut{\X}\le \Aut{\Y}$.}

We denote the Hamming scheme of order $m^d$ and rank $d+1$ by $\cH(d, m)$
and the Johnson scheme of order $\binom{m}{k}$ and rank $k+1$ by $\cJ(m, k)$.
The special case $\cH(1, m) \cong \cJ(m, 1)$ is the \emph{trivial scheme}, denoted $\cT_m$.
The $d$th tensor power of an association scheme $\X$ is denoted $\X^d$.
The symmetrized $d$th tensor power of the Johnson scheme $\cJ(m, k)$ is called the \emph{Cameron scheme} and denoted $\cC(m, k, d)$.

In this paper we classify primitive fusions of $\cJ(m, k)^d$ assuming $m$ is sufficiently large in terms of $k$ and $d$.

\begin{theorem}
    \label{main-thm}
    \textcolor{blue}{For any positive integers $k,d$ there exists a constant $m_0(k,d)$ such that
    any primitive fusion $\X$ of $\cJ(m, k)^d$ with $m \geq m_0(k, d)$ belongs, up to permuting coordinates, to one of the following intervals:
    \begin{enumerate}
        \item $\cJ(m, k)^d \leq \X \leq \cC(m, k, d)$,  
        \item $\T_{M^e}^{d/e} \leq \X \leq \cH(d/e, M^e)$ for some integer $e \mid d$ where $M=\binom{m}{k}$.
    \end{enumerate}
    }
\end{theorem}

The special cases $k = 1$ and $d = 1$ are worth highlighting individually.
In \cite{M-hamming} it was shown that $\cH(d, m)$ has no nontrivial fusions for $m > 4$. The case $k=1$ of \Cref{main-thm} more generally classifies primitive fusions of $\T_m^d$ (for $m$ sufficiently large).

\begin{cor}
    Let $\X$ be a primitive fusion of $\T_m^d$, where $m \geq m_0(1,d)$.
    Then, up to permuting coordinates, $\T_{m^e}^{d/e} \le \X \le \cH(d/e, m^e)$
    for some integer $e \mid d$.
\end{cor}

In \cite{M-johnson} it was shown that $\cJ(m, k)$ has no nontrivial fusions for $m \geq 3k+4$.
\textcolor{blue}{This result was improved  to $m\geq 3k-1$ in \cite{U}.}
The case $k = 1$ of \Cref{main-thm} recovers this result, except for the precise lower bound.

\begin{cor}
    Let $\X$ be a fusion of $\cJ(m, k)$, where $m \geq m_0(k,1)$. Then either $\X = \cJ(m, k)$ or $\X$ is trivial.
\end{cor}

Association schemes of the type appearing in the conclusion of \Cref{main-thm} are studied in \cite{E}, where they are called ``Cameron sandwiches'' and ``Hamming sandwiches'', respectively.
The main result of \cite{E} is that there are infinite families of nonschurian Hamming sandwiches.

\begin{remark}
Imprimitive fusions of $\cJ(m, k)^d$ are not so easily classified.
Certainly one must allow arbitrary tensor products of the cases appearing in \Cref{main-thm}, but there are still many others.
For example, the imprimitive wreath product $\T_m \wr \T_m$ (see \cite{CP}*{Section~3.4.1}) is an imprimitive fusion of $\T_m^2$ not fitting this description.
\end{remark}

As an application we give an elementary classification of primitive groups containing $(A_m^{(k)})^d$.
Here $A_m^{(k)}$ denotes the image of the alternating group $A_m$ in its permutation action on $k$-sets,
and below $S_m^{(k)}$ is defined similarly.

\textcolor{blue}{The statement below is a special case of Cameron's theorem~\cites{cameron,liebeck,maroti},
which more generally classifies all large primitive permutation groups.
However, while the proof of Cameron's theorem depends on the classification of finite simple groups, our proof does not.} 
\begin{cor}
    \label{cor:cameron}
    Let $G \le S_n$ be a primitive permutation group containing $(A_m^{(k)})^d$,
    where $n = \binom{m}{k}^d$ and $m \ge m_0(k, d)$.
    Then either
    \begin{enumerate}
        \item $(A_m^{(k)})^d \le G \le S_m^{(k)} \wr S_d$ or
        \item $(A_{M^e})^{d/e} \le G \le S_{M^e} \wr S_{d/e}$ for some integer $e \mid d$ where $M = \binom{m}{k}$.
    \end{enumerate}
\end{cor}
\begin{proof}
\textcolor{blue}{ Let $\X$ be the orbital scheme of $G$. It follows from $(A_m^{(k)})^d\leq G$ that $\X$ is a fusion of the orbital scheme of $(A_m^{(k)})^d$, which coincides with $\cJ(m,k)^d$.
Thus $\cJ(m,k)^d\leq \X$ and, by~\Cref{main-thm}, either $\cJ(m,k)^d\leq \X \leq \cC(m,k,d)$ or $\T_{M^e}^{d/e} \leq \X \leq \cH(d/e, M^e)$.
In the first case, $\X$ has a constituent graph equal to the Cameron graph $C(m, k, d)$.
In the second case, $\X$ has a consituent graph equal to the Hamming graph $H(d/e, M^e)$ (which is a special case of a Cameron graph).
Applying \cite{wilmes-thesis}*{Theorem~8.2.1}, either the claimed conclusion holds or $G$ is small: $|G| \le \exp(c (\log n)^3)$.
Since $|G| \ge |(A_m^{(k)})^d| = m!^d$ and $n = \binom{m}{k}^d \le m^{kd}$,
we get $d \log m! \le c (\log n)^3 \le c (kd \log m)^3$, in contradiction to the hypothesis $m \ge m_0(k, d)$.
}
\end{proof}

\section{Notation}

Let $[0,k]^d = \{0,\dots,k\}^d$.
We use the following notation for vectors $a,b,c \in [0,k]^d$:
\begin{align*}
    &a \leq b \iff a_i \leq b_i ~ \text{for all} ~ i ~ \textcolor{blue}{(\text{in this case we say that } b ~ \text{\emph{dominates}} ~ a)}, &\\
    &|a - b| = (|a_1 - b_1|, \dots, |a_d - b_d|),\\
    &\min(a, b) = (\min(a_1, b_1), \dots, \min(a_d, b_d)),\\
    &\max(a, b) = (\max(a_1, b_1), \dots, \max(a_d, b_d)),\\
    &a! = a_1! \cdots a_d!,\\
    &\binom{a}{b} = \binom{a_1}{b_1} \cdots \binom{a_d}{b_d},\\
    &\wt(a) = a_1 + \cdots + a_d, \\
    &\supp(a) = \{ i : a_i > 0\}, \\
    &[a] = \{ b \in [0,k]^d : b \leq a\}, \\
    &(x)^d = (x, \dots, x), \\
    &e_i = (0, \dots, 0, 1, 0, \dots, 0).
\end{align*}
We understand $\binom{a}{b}$ to be zero unless $(0)^d \le b \le a$.
We call $\wt(a)$ the \emph{weight} of $a$ and $\supp(a)$ the \emph{support} of $a$.

If $p$ is a polynomial in one variable we write $\deg(p)$ for its degree and $\lt(p)$ for its leading term.

\section{Proof}

The structure constants of $\cJ(m,k)$ are given by
\[
    p_{b,c}^a(m) = \sum_{i} \binom{k-a}{i}\binom{a}{k-b-i}\binom{a}{k-c-i} \binom{m-k-a}{b+c+i-k}
\]
(as in \cites{E,KK}).
Here $0 \leq a, b, c \leq k$, and $i$ can be restricted to the range
\[
    \max(0,k-a-b,k-b-c,k-a-c) \leq i \leq \min(k-a,k-b,k-c, m-a-b-c).
\]
In the following we always assume $m \geq 3k$. Under this assumption we have $p_{b,c}^a(m) > 0$ if and only if $a,b,c$ satisfy the triangle inequalities \cite{E}*{Lemma~4.1}.
More precisely we have the following.

\begin{lemma}
    We have $p^a_{b, c}(m) > 0$ if and only if $|b - c| \leq a \leq b + c$. 
    Assuming $0 \leq b - c \leq a \leq b + c$, the leading term of $p^a_{b,c}(m)$ is
    \[
        \lt(p_{b,c}^a(m)) = \begin{cases}
        \binom{a}{b}\binom{a}{c}\frac{1}{(b+c - a)!} m^{b+c-a} &: a \geq b,\\
        \binom{k-a}{k-b}\binom{a}{b-c}\frac{1}{c!} m^{c} &: a\leq b.
        \end{cases}
    \]
In particular, $\deg(p_{b,c}^a(m))\leq \min(b,c)$, and equality holds if and only if $a \leq \max(b,c)$.
\end{lemma}

Now consider $\cJ(m, k)^d$.
For $a, b, c \in [0,k]^d$, let
\[
    p^a_{b,c}(m) = \prod_{i=1}^d p^{a_i}_{b_i,c_i}(m).
\]
These are the structure constants of $\cJ(m, k)^d$.
The following lemma generalizes the previous one.

\begin{lemma}
    Let $a,b,c\in[0,k]^d$.
    Then $p^a_{b,c}(m) > 0$ if and only if $|b-c| \leq a \leq b + c$ 
    Moreover $\deg(p_{b,c}^a(m)) \leq \wt(\min(b, c))$, with equality if and only if $a \leq \max(b, c)$.
    If $\deg(p_{b,c}^a(m)) = \wt(b) = \wt(c)$ then $a \leq b = c$ and the leading term of $p^a_{b,c}(m)$ is
    \[
        \lambda(p^a_{b,c}(m))
        = \binom{(k)^d - a}{(k)^d - b} \frac{1}{b!} m^{\wt(b)}.
    \]
\end{lemma}
\begin{proof}
    It follows from $p_{b,c}^a(m) = \prod_{i=1}^d p_{b_i c_i}^{a_i}(m)$ that $p_{b,c}^a(m)>0$ iff each triple $a_i,b_i,c_i$ satisfies the triangle condition $|b_i-c_i|\leq a_i \leq b_i+c_i$. In this case
    \[
        \deg(p_{b,c}^a(m)) = \sum_{i=1}^d \deg(p_{b_i, c_i}^{a_i}(m))\leq \sum_{i=1}^d\min(b_i,c_i) = \wt(\min(b, c)).
    \]
    Equality holds if and only if $a_i \leq \max(b_i, c_i)$ for all $i$.
    
    If $\deg(p_{b,c}^a(m)) = \wt(b) = \wt(c)$ then $\wt(\min(b, c)) = \wt(b) = \wt(c)$, which implies $b = c$, and moreover we have seen that we must have $a \leq b$.
    Multiplying the leading terms of $p^{a_i}_{b_i,c_i}(m)$ given by the previous lemma, we get the claimed formula.
\end{proof}

Let $\X$ be a fusion of $\cJ(m, k)^d$.
Since $\X$ is a coarsening of $\cJ(m, k)^d$ there is a partition $\fS$ of $[0,k]^d$ such that $\X = \{R_\alpha : \alpha \in \fS\}$, where
\[
    (u, v) \in R_\alpha \iff (|u_1 \setminus v_1|, \dots, |u_d \setminus v_d|) \in \alpha.
\]
In this situation we write $\X = \cJ(m, k)^\fS$ \cite{E}*{Section~4}.
For $\beta, \gamma \in \fS$ and $a \in [0,k]^d$ define
\[
    p^a_{\beta,\gamma}(m) = \sum_{b\in \beta, c \in \gamma} p^a_{b,c}(m).
\]
For $\X = \cJ(m,k)^\fS$ to be an association scheme, $\fS$ must satisfy two conditions:
\begin{enumerate}
    \item $\{(0)^d\} \in \fS$,
    \item $p^a_{\beta,\gamma}(m) = p^{a'}_{\beta,\gamma}(m)$ for all $\alpha,\beta,\gamma \in \fS$ and $a,a' \in \alpha$.
\end{enumerate}
We may denote the common value of $p^a_{\beta,\gamma}(m) ~ (a\in\alpha)$ by $p^\alpha_{\beta,\gamma}(m)$; these are the structure constants of $\X$.

We call the sets $\alpha \in \fS$ the \emph{basic $\fS$-sets};
their unions are called \emph{$\fS$-sets}.
For nonempty $S \subseteq [0,k]^d$ let $\wt(S) = \max\{\wt(b) \mid b \in S\}$.
For $\alpha \in \fS$ let $\alpha^*$ be the set of $a \in \alpha$ of maximal weight.
Let $D_\alpha = \bigcup_{a \in \alpha^*} [a]$. Note that $\alpha^*\subseteq D_\alpha$.

\textcolor{blue}{For any $\alpha,\beta\subseteq [0,k]^d$ and $a\in [0,k]^d$ the structure constant $p_{\alpha,\beta}^a(m)$ is a real polynomial in $m$ the coefficients of which depend on $\alpha,\beta$ and $a$. For every pair of distinct real polynomials $f,g\in \mathbb{R}[x]$ there exists a real number $c_{f,g}\in\mathbb{R}$ such that $f(x)\neq g(x)$ holds for all $x > c_{f,g}$. Therefore, there exists a constant $m_0(k,d)$ such that, provided $m \geq m_0(k, d)$, the following condition holds for all $\alpha,\beta\subset [0,k]^d$ and $a,b\in [0,k]^d$:
\begin{equation}\label{poly}
    p_{\alpha,\beta}^a(m) = p_{\alpha,\beta}^b(m)
    \implies p_{\alpha,\beta}^a(m) = p_{\alpha,\beta}^b(m)\text{ as polynomials in } m
    \implies \lambda(p_{\alpha,\beta}^a(m)) = \lambda(p_{\alpha,\beta}^b(m)).
\end{equation}}

\textcolor{blue}{In what follows we assume that $m \geq m_0(k, d)$ and hence $p^a_{\beta,\gamma}(m) = p^{a'}_{\beta,\gamma}(m)$ only if $p^a_{\beta,\gamma}(m)$ and $p^{a'}_{\beta,\gamma}(m)$ are equal as polynomials in $m$.}

\begin{prop}\label{key-prop}
    Let $\beta \in \fS$ be a basic $\fS$-set.
    \begin{enumerate}
        \item $D_\beta$ is an $\fS$-set, and $\wt(D_\beta \setminus \beta) < \wt(\beta)$.
        \item For every basic set $\alpha \in \fS$ there is a constant $N^\beta_\alpha$ such that
        \[
            \sum_{b : a \leq b \in \beta^*} \binom{(k)^d - a}{(k)^d - b} = N^\beta_\alpha \qquad (a \in \alpha).
        \]
        \item Every element of $\beta$ is dominated by a unique element of $\beta^*$.
        \item Either $\wt(\beta) = \wt(D_\beta \setminus \beta) + 1$ or $\beta^* \subset \{0, k\}^d$.
    \end{enumerate}
\end{prop}
\begin{proof}
    Let $w = \wt(\beta)$ and $D = D_\beta$.
    Since the polynomials $p^a_{b,c}(m)$ have positive leading coefficient whenever they are nonzero,
    the degree of $p^a_{\beta,\beta}(m)$ is, by the previous lemma,
    \[
        \deg(p^a_{\beta,\beta}(m)) = \max_{b, c \in \beta} \deg(p^a_{b,c}(m)) \le \max_{b, c \in \beta} \wt(\min(b, c)) \le w,
    \]
    with equality holding if and only if $a \leq b = c \in \beta^*$, i.e., if and only if $a \in D$.
    Since $p^a_{\beta,\beta}(m)$ should depend only on the cell of $\fS$ containing $a$,
    it follows that $D$ is an $\fS$-set.
    Since $\beta$ is basic, $\beta \subset D$, and since $\beta \supset \beta^*$ we have $\wt(D \setminus \beta) < \wt(\beta)$.
    
    Moreover if $a \in D$ then the leading term of $p^a_{\beta,\beta}$ is
    \[
        \lambda(p^a_{\beta,\beta}(m)) = \sum_{b : a \leq b \in \beta^*} \binom{(k)^d - a}{(k)^d - b} \frac1{b!} m^w,
    \]
    and again this should depend only on the cell $\alpha$ containing $a$.
    \textcolor{blue}{Taking $\alpha = \beta$ and $a\in \beta^*$, it follows that $a!$ is a constant for $a \in \beta^*$.}
    Hence (2) holds (if $\alpha$ is not contained in $D$ then $N^\beta_\alpha = 0$).
    
    Next apply (2) with $\alpha = \beta$.
    Taking $a \in \beta^*$ shows $N^\beta_\beta = 1$, so (3) holds.
    
    Finally let $b \in \beta^*$ and suppose $a = b - e_i \ge 0$.
    If $a \in \beta$ then we get
    \[
        \binom{k-b_i+1}{k-b_i} = N^\beta_\beta = 1,
    \]
    so $b_i = k$.
    Hence either $\wt(b) = \wt(\beta)+1$ or $b_i \in \{0, k\}$ for all $i$, which implies (4).
\end{proof}

We can define a partial ordering on $\fS$ by saying $\alpha \preceq \beta$ if every $a \in \alpha$ is dominated by some $b \in \beta$.
By part (3) of the proposition, this is equivalent to $N^\beta_\alpha > 0$.
It is obvious that $\preceq$ is reflexive and transitive on $\fS$.
To verify antisymmetry, note that if $\alpha \preceq \beta \preceq \alpha$ and $a \in \alpha^*$ then there are $b \in \beta$ and $a' \in \alpha$ such that $a \le b \le a'$, which by maximality of $a$ implies $a = b = a'$ and hence $\alpha = \beta$ since $\fS$ is a partition.
We say $\alpha \in \fS$ is \emph{minimal} if it is minimal in $(\fS \setminus \{(0)^d\}, \preceq)$.

\begin{cor}
    Let $\alpha \in \fS$ be a minimal basic set.
    Then $D_\alpha = \alpha \cup \{(0)^d\}$.
    Moreover the elements of $\alpha^*$ have disjoint equal-sized supports,
    and either $\wt(\alpha) = 1$ or $\alpha^* \subset \{0, k\}^d$.
\end{cor}
\begin{proof}
    Apply the proposition.
    Note that if $\beta$ is a basic subset of $D_\alpha \setminus \alpha$ then $\beta \prec \alpha$.
    By minimality of $\alpha$ this implies $\beta = \{(0)^d\}$.
    Hence $D_\alpha = \alpha \cup \{(0)^d\}$.
    Next,
    for any $a \in \alpha^*$ and $i \in \supp(a)$ we have $e_i \in D_\alpha$.
    By part (3) of \Cref{key-prop}, $a$ is the unique element of $\alpha^*$ dominating $e_i$.
    Therefore the elements of $\alpha^*$ have disjoint supports.
    By part (4), either $\wt(\alpha) = 1$ or $\alpha^* \subset \{0, k\}^d$.
    If $\wt(\alpha) = 1$ then all $a \in \alpha^*$ have singleton support,
    and otherwise $|\supp(a)| = \wt(a) / k = \wt(\alpha) / k$ for all $a \in \alpha^*$.
\end{proof}

Until now $\X$ could be imprimitive.
Now we specialize to the primitive case to complete the proof of \Cref{main-thm}.
Let $\alpha \in \fS$ be minimal.
If $\X$ is primitive then $R_\alpha$ must be connected,
which implies that $\alpha^*$ covers $\{1, \dots, d\}$.
Hence $\alpha^*$ is an equipartition of $\{1, \dots ,d\}$.
Let $w = \wt(\alpha)$.
If $w = 1$ then $\alpha^*$ must be the set of elements of weight $1$, so $R_\alpha$ is the Cameron graph.
Since the Weisfeiler--Leman stabilization of the Cameron graph is the Cameron scheme, we find $\X \le \cC(m, k, d)$.
If $w > 1$ then $\alpha^* \subset \{0, k\}^d$.
If the elements of $\alpha^*$ have support size $e$ 
then $R_\alpha$ is the Hamming graph $H(M^e, d/e)$, so $\X \le \cH(d/e, M^e)$.
To finish we must show $\cT_{M^e}^{d/e} \le \X$.
For this it suffices to prove that for every $\beta \in \fS$, the elements of $\beta^*$ are sums of elements of $\alpha^*$ (and in particular $\beta^* \subset \{0, k\}^d$).

We apply \Cref{key-prop}(2) to $\alpha$ and $\beta$.
Let $i \in \{1, \dots, d\}$.
Taking $a = e_i$, we find that $N^\beta_\alpha$ is at least the number of $b \in \beta^*$ such that $b_i > 0$.
On the other hand, taking $a$ to be the unique element of $\alpha^*$ dominating $e_i$,
since $a \in \{0,k\}^d$ we find that $N^\beta_\alpha$ is equal to the number of $b \in \beta^*$ such that $a \le b$.
Hence $b_i > 0$ implies $a \le b$.
This implies that $b$ is the sum of those $a \in \alpha^*$ such that $a \le b$, as required.

\bibliography{refs}
\end{document}